\documentclass[10pt,twocolumn,amsmath,amssymb,aps,pra,secnumarabic,%
    nofootinbib]{revtex4-1}
\usepackage{mathptmx,amsthm,mathdots,graphics,tikz}
\newtheorem{theorem}{Theorem}[section]
\newtheorem{proposition}[theorem]{Proposition}
\newtheorem{lemma}[theorem]{Lemma}

\theoremstyle{remark}
\newtheorem*{remark}{Remark}

\makeatletter\def\@bibdataout@init{}\def\pre@bibdata{}\makeatother

\newcommand{\BQP}{\mathsf{BQP}}

\newcommand{\FNP}{\mathsf{FNP}}
\newcommand{\NP}{\mathsf{NP}}

\newcommand{\shP}{\mathsf{\#P}}
\newcommand{\coNP}{\mathsf{coNP}}

\newcommand{\Z}{\mathbb{Z}}
\newcommand{\Q}{\mathbb{Q}}
\newcommand{\C}{\mathbb{C}}

\newcommand{\tM}{\tilde{M}}
\newcommand{\del}{\partial}
\newcommand{\Ex}{\mathrm{Ex}}
\newcommand{\GL}{\mathrm{GL}}
\newcommand{\longto}{\longrightarrow}

\newcommand{\defeq}{\stackrel{\mathrm{def}}{=}}

\newcommand{\Thm}[1]{Theorem~\ref{#1}}
\newcommand{\Lem}[1]{Lemma~\ref{#1}}
\newcommand{\Sec}[1]{Section~\ref{#1}}
\newcommand{\Prop}[1]{Proposition~\ref{#1}}

\newcommand{\Fig}[1]{Figure~\ref{#1}}
\newcommand{\ie}{\emph{i.e.}}
\newcommand{\eg}{\emph{e.g.}}

\newenvironment{eq}[1]{\begin{equation}\label{#1}}
    {\end{equation}\ignorespacesafterend}

\begin{document}

\title{Identifying lens spaces in polynomial time}

\author{Greg Kuperberg}
\email{greg@math.ucdavis.edu}
\thanks{Partly supported by NSF grants CCF-1319245 and CCF-1716990.}
\affiliation{University of California, Davis}

\begin{abstract} We show that if a closed, oriented 3-manifold $M$ is
promised to be homeomorphic to a lens space $L(n,k)$ with $n$ and $k$
unknown, then we can compute both $n$ and $k$ in polynomial time in the
size of the triangulation of $M$.  The tricky part is the parameter $k$.
The idea of the algorithm is to calculate Reidemeister torsion using
numerical analysis over the complex numbers, rather than working directly
in a cyclotomic field.
\end{abstract}

\maketitle

\section{Introduction}
\label{s:intro}

The algorithmic problem of distinguishing or classifying closed
$d$-dimensional manifolds is elementary when $d \le 2$, provably impossible
when $d \ge 4$, and recursive when $d=3$ \cite{K:homeo}.  The remaining
question is how efficiently we can distinguish closed 3-manifolds; or
whether we can distinguish them efficiently with one or another form of
help.  One small but interesting part of this question is the case of lens
spaces.  If $M$ is a closed, oriented 3-manifold, conventionally given by a
triangulation, then is it a lens space?   If so, which one?  In this article,
we show that at least the second question has an efficient algorithm.

\begin{theorem} Suppose that $M$ is a closed, oriented 3-manifold given by
a triangulation with $t$ tetrahedra, and that we are promised
that $M \cong L(n,k)$ is a lens space with $n$ and $k$ unknown.
Then $n$ and $k$ can be computed in deterministic polynomial time in $t$.
\label{th:main} \end{theorem}

The motivation for our result is a recent result announced by Lackenby and
Schleimer \cite{LS:lens} to both recognize whether $M$ is a lens space,
and if so which one, in the complexity class $\FNP$.  In other words, they
provide a deterministic algorithm (a verifier) with the help of a prover
who asserts the answer and provides a certificate that it is correct.
Thus, \Thm{th:main} implies that in the Lackenby-Schleimer result, it is
enough for the prover to only provide a certificate that $M$ is a lens
space at all, which is simpler.  According to Lackenby and Schleimer,
the certificate can be a Heegaard torus which is almost normal relative
to the triangulation of $M$.

Recall that the standard lens space $L(n,k)$ is constructed by gluing the
top hemisphere of a ball, often imagined as a convex dihedron or ``lens",
to the bottom hemisphere with a rotation of $2\pi k/n$.  The calculation
of $n$ is reasonably standard, because if $M \cong L(n,k)$, then we can
calculate the homology $H_1(M) \cong \Z/n$ in polynomial time using a
version of the Smith normal form algorithm \cite{KB:smith}.  The second
parameter $k$ is more subtle.  We can take it to be a prime residue $k \in
(\Z/n)^\times$.  Reidemeister \cite{Reidemeister:lins} showed that
\[ L(n,k_1) \cong L(n,k_2) \]
as oriented 3-manifolds if and only if $k_1 = k_2$ or $k_1 = 1/k_2$.

In another respect, both parameters are more subtle than one might expect.
Suppose that $M \cong L(n,k)$ has $t$ tetrahedra.  In the most standard
(generalized) triangulation of $L(n,k)$, $n = t$.   But there are
other families of triangulated manifolds $M \cong L(n,k)$ such that
$n$ is exponential in $t$, and with exponentially many values of $k$
for specific values of $n$.  See \Sec{s:large}.  If we can be promised a
polynomial bound on $n$ itself rather than merely a polynomial bound on its
digits, then it is easier to calculate $k$, because we can directly follow
Reidemeister's method by computing the Reidemeister torsion $\Delta$ of $M$
(endowed with a suitable local system of coefficients) in the cyclotomic
ring $\Z[\zeta_n]$ or its fraction field $\Q(\zeta_n)$, where $\zeta_n$
is a primitive $n$th root of unity.

The idea of our proof of \Thm{th:main} is to approximately compute the
Reidemeister torsion using numerical analysis over the complex
numbers $\C$.  If we let $\zeta_n = \exp(2\pi i/n)$, the result is a sparse
polynomial expression
\[ \Delta = \zeta_n^c (1-\zeta_n^a)(1-\zeta_n^b) \in \C. \]
In order to establish a polynomial-time algorithm, we want a polynomial
upper bound on the digits of precision of an approximation to $\Delta$
that we need to resolve the exponents $a$, $b$, and $c$.  We also need an
algorithm to calculate those exponents.  More precisely, the precision bound
needs to be polynomial in $t$ and thus polynomial in $\log(n)$.  According to
MathOverflow\footnote{\url{http://mathoverflow.net/questions/46068}},
even the first part is unknown for general sparse sums of roots of unity.
A bound is known for sums with at most four terms \cite{Myerson:small}.
(Remark: The unproven behavior of sparse sums of powers of $\zeta_n$
can be circumvented by making $\zeta_n$ a randomly chosen primitive
$n$th root of unity rather than specifically $\exp(2\pi i/n)$.)
More to the point, the precision problem is easier in our case, and
we can also solve for the exponents with the aid of another answer in
MathOverflow\footnote{\url{http://mathoverflow.net/questions/215852}}.

In a previous version of this paper \cite{K:lensv1}, the author found a
weaker version of \Thm{th:main} with a quantum polynomial-time algorithm,
\ie, an algorithm in $\BQP$ \cite{NC:qcqi}.  The idea then was to replace
$\Z[\zeta_n]$ with a quotient field $\Z/p$, where $p$ is a prime which is
congruent to $1$ mod $n$.  Then the Reidemeister torsion calculation reduces
to the discrete logarithm problem, which can be solved with Shor's algorithm
\cite{Shor:factorization}.  A quantum algorithm which is faster than any
competing classical algorithm is always interesting, but in this case the
author later noticed that there is a fast classical algorithm after all.

The question remains whether there is a competitive quantum algorithm for any
natural question in 3-manifold topology.  This is a natural thing to look
for, since for instance it is known that unknottedness is the complexity
class $\NP \cap \coNP$ \cite{HLP:complexity,K:knottedness,Lackenby:norm}.
(See the Complexity Zoo \cite{W:zoo} for a survey of computational complexity
classes.)  While $\NP \cap \coNP$ is thought to neither contain nor be
contained in quantum polynomial time $\BQP$, some key problems (such as
discrete logarithm) are known to be in both of them.  Note that Aharonov,
Jones, and Landau \cite{AJL:approx} give an algorithm to approximate the
Jones polynomial of a knot at a principal root of unity; this algorithm also
has a version for 3-manifolds \cite{GMR:three}.  However, the approximation
is exponentially poor; any fair approximation that could be useful
for geometric topology is $\shP$-hard \cite{K:jones}.

\acknowledgments

The author would like to thank an anonymous MathOverflow user for help
with part of the calculation.

\section{Large lens spaces with small triangulations}
\label{s:large}

\begin{figure*}[htb]\begin{center}\begin{tikzpicture}[semithick]
\fill[white!85!blue] (0,1.6) arc (90:270:1.6) -- (8,-1.6)
    arc (-90:90:1.6) -- (0,1.6);
\draw (5,1.6) -- (0,1.6) arc (90:270:1.6) -- (5,-1.6);
\draw (7,-1.6) -- (8,-1.6) arc (-90:90:1.6) -- (7,1.6);
\draw[dotted] (5,1.6) -- (7,1.6) (5,-1.6) -- (7,-1.6);
\draw (-1,0) node {$\sigma_1$};
\draw (1,0) node {$\tau_{a_1}$};
\draw (.8,.7) edge[bend right=30,->] (1.2,1.3);
\draw (.8,-1.3) edge[bend right=30,->] (1.2,-.7);
\draw (3,0) node {$\tau_{a_2}$};
\draw (2.8,1.3) edge[bend left=30,->] (3.2,.7);
\draw (2.8,-.7) edge[bend left=30,->] (3.2,-1.3);
\draw (5,0) node {$\tau_{a_3}$};
\draw (4.8,.7) edge[bend right=30,->] (5.2,1.3);
\draw (4.8,-1.3) edge[bend right=30,->] (5.2,-.7);
\draw (6.8,.7) edge[bend right=30,->] (7.2,1.3);
\draw (6.8,-1.3) edge[bend right=30,->] (7.2,-.7);
\draw (7,0) node {$\tau_{a_m}$};
\draw (6,0) node {\scalebox{1.5}{$\cdots$}};
\draw (9,0) node {$\sigma_2$};
\begin{scope}[shift={(0,0)},xscale=.5]
\draw (0,0) ellipse (.9 and 1.6);
\draw (-.08,-.88) -- (0,-.72) .. controls (.16,-.4) and (.18,-.16) .. (.18,0)
    .. controls (.18,.16) and (.16,.4) .. (0,.72) -- (-.08,.88);
\draw (0,-.72) .. controls (-.16,-.4) and (-.2,-.16) .. (-.2,0)
    .. controls (-.2,.16) and (-.16,.4) .. (0,.72);
\end{scope}
\begin{scope}[shift={(2,0)},xscale=.5]
\draw (0,0) ellipse (.9 and 1.6);
\draw (-.08,-.88) -- (0,-.72) .. controls (.16,-.4) and (.18,-.16) .. (.18,0)
    .. controls (.18,.16) and (.16,.4) .. (0,.72) -- (-.08,.88);
\draw (0,-.72) .. controls (-.16,-.4) and (-.2,-.16) .. (-.2,0)
    .. controls (-.2,.16) and (-.16,.4) .. (0,.72);
\end{scope}
\begin{scope}[shift={(4,0)},xscale=.5]
\draw (0,0) ellipse (.9 and 1.6);
\draw (-.08,-.88) -- (0,-.72) .. controls (.16,-.4) and (.18,-.16) .. (.18,0)
    .. controls (.18,.16) and (.16,.4) .. (0,.72) -- (-.08,.88);
\draw (0,-.72) .. controls (-.16,-.4) and (-.2,-.16) .. (-.2,0)
    .. controls (-.2,.16) and (-.16,.4) .. (0,.72);
\end{scope}
\begin{scope}[shift={(8,0)},xscale=.5]
\draw (0,0) ellipse (.9 and 1.6);
\draw (-.08,-.88) -- (0,-.72) .. controls (.16,-.4) and (.18,-.16) .. (.18,0)
    .. controls (.18,.16) and (.16,.4) .. (0,.72) -- (-.08,.88);
\draw (0,-.72) .. controls (-.16,-.4) and (-.2,-.16) .. (-.2,0)
    .. controls (-.2,.16) and (-.16,.4) .. (0,.72);
\end{scope}
\end{tikzpicture}\end{center}
\caption{A lens space $L(n,k)$ as solid tori with triangulations
    $\sigma_1$ and $\sigma_2$, connected by twisted bundles $(S^1 \times S^1)
    \rtimes I$ with triangulations $\tau_{a_j}$.}
\label{f:hard} \end{figure*}
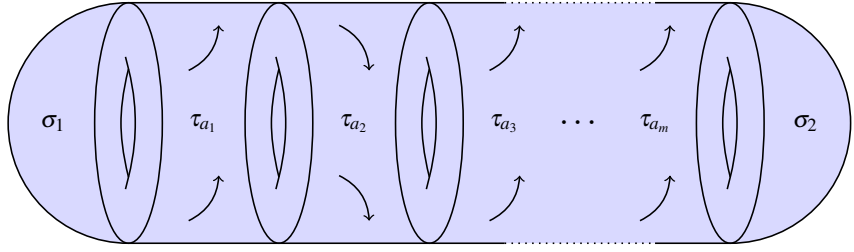

In this section, we will construct lens spaces $M \cong L(n,k)$ where
$n$ is much larger than the number of tetrahedra $t$, and $k$ has many
possible values.  The manifolds that we construct are easy to identify
given their specific triangulations.  However, the triangulations can
then be obfuscated with local moves (\eg, Newman-Pachner bistellar moves).
\Prop{p:large} makes both \Thm{th:main} and the Lackenby-Schleimer result
look more interesting.   For the latter question, it is easy to compute
whether $H_1(M) \cong \Z/n$ is cyclic.  If it is, and if $n$ is polynomially
bounded in $t$, then Schleimer's prior result \cite{Schleimer:sphere}
gives an algorithm in $\NP$ to compute whether the abelian cover $\tM$
is homeomorphic to $S^3$, which then implies that $M$ is a lens space.

\begin{proposition} There exists a family of triangulated lens spaces $\{M
\cong L(n,k)\}$ with $t = t(n,k)$ tetrahedra, such that $n$ is exponential
in $t$ and there are exponentially many choices for $k$ for each fixed $n$.
\label{p:large} \end{proposition}

\begin{proof} Our construction is equivalent to a well-known
construction of lens spaces using Dehn surgery on a chain of unknots
\cite[Ex. 9H13]{Rolfsen:knots}.

We choose a fixed triangulation $\sigma$ of the torus $T = S^1 \times
S^1$, and we choose two solid tori $X_1, X_2$ with $\del X_1, \del X_2
= T$, and with triangulations $\sigma_1,\sigma_2$ that extend $\sigma$.
We can describe an element of the mapping class group of $T$ by an element
of $\GL(2,\Z)$ that describe its action on the homology group $H_1(T)$.
For each $1 \le a \le 5$, we choose a fixed triangulation $\tau_a$
of a torus bundle over an interval, $T \rtimes I$, that connects the
triangulation $\sigma$ of $T$ to itself using the monodromy matrix
\[ F_a = \begin{pmatrix} 0 & 1 \\ 1 & a \end{pmatrix}. \]
Our construction is to concatenate a sequence $\{\tau_{a_j}\}_{1 \le j
\le m}$ of these mapping cylinders together with a solid torus at each
end, as in \Fig{f:hard}.  We also assume that $a_1 > 1$.  The tetrahedron
number $t$ is thus $O(m)$.  If the solid tori $\sigma_1$ and $\sigma_2$
are positioned suitably, then the result is $M \cong L(n,k)$, where $n$
and $k$ are given as a finite continued fraction:
\[ \frac{n}{k} = a_m + \frac1{a_{m-1} +
    \frac{1}{\ddots_{+ \frac1{a_1}}}}. \]
If we let $n_j/k_j$ be the $j$th partial evaluation, then
we can also express the calculation with the recurrence
\[ k_j = n_{j-1} \qquad n_j = a_j n_{j-1} + k_j = a_j n_{j-1} + n_{j-2}. \]
The answer $n/k$ determines the monodromy numbers $\{a_j\}$ since the
continued fraction is unique under the constraint $a_1 > 1$.  Since the
integers $\{n_j\}$ increase, we obtain the inequality
\[ n_j < (a_j+1)n_{j-1}.\]
If we choose the sequence of monodromy numbers at random, we obtain the
probabilistic relation
\[ \Ex[\log(n_j)] < \Ex[\log(a_j+1)] + \Ex[\log(n_j)]. \]
Also,
\begin{align*}
\Ex[\log(a_j+1)] &= \frac{\log(2) + \log(3) + \dots + \log(6)}5 \\
    & < \log(3.73).
\end{align*}
By the law of large numbers, most monodromy sequences produce $n < 3.73^m$.
On the other hand, there are $4 \cdot 5^{m-1}$ sequences of length $m$, so
by the pigeonhole principle, some value of $n$ must see exponentially many
values of $k$.  Any such value of $n$ must also be exponentially large.
In any case, for every choice of numbers $\{a_j\}$, $\{n_j\}$ grows at
least as fast as the Fibonacci numbers, which also implies that $n$ is
exponentially large.
\end{proof}

\section{Reidemeister torsion}
\label{s:torsion}

We review Reidemeister torsion \cite{Turaev:torsions} and its
value for lens spaces.

Suppose that
\[ C_* = \{C_k \stackrel{\del}\longto C_{k-1}\}_{0 \le k \le m} \]
is a finite, acyclic chain complex over a field $F$.  (Reidemeister torsion
is well defined for a free complex over any commutative ring, but it is
easier to discuss algorithms in the field case.)  Suppose in addition
that each term $C_k$ has a distinguished basis.  Since $C_*$ is acyclic
and finite, it is isomorphic to a direct sum of complexes of the form
\[ 0 \longto F \stackrel{\cong}{\longto} F \longto 0. \]
Define an \emph{adapted basis} $A_*$ for $C_*$ to be
one induced by such a decomposition.   In other words, if $\alpha_k \in
A_k$ is a basis vector, then either $\del \alpha_k = 0$, or $\del \alpha_k
\in A_{k-1}$ is another basis vector. Then the \emph{Reidemeister torsion}
of $C_*$ is
\[ \Delta(C_*) \defeq (\det A_0) (\det A_1)^{-1} (\det A_2)
    \cdots (\det A_m)^{(-1)^m}, \]
where each $A_j$ is also interpreted as the change-of-basis matrix from
the distinguished basis to the adapted basis.  The following two facts
are standard:
\begin{description}
\item[1.] Every adapted basis yields the same value of $\Delta(C_*)$.
\item[2.] Let $C_*$ be the chain complex of a finite CW complex $\sigma$
with PL attaching maps, possibly with twisted coefficients, and using
the cells of $\Phi$ as its distinguished basis.  Then the Reidemeister
torsion $\Delta(C_*)$ is invariant under refinement of $\Phi$.
\end{description}
The second fact essentially says that Reidemeister torsion is a PL
topological invariant.  We have to be careful because the sign of
$\Delta(C_*)$ depends on the ordering and orientation of the cells of
$\Phi$, and ambiguities in the local coefficient system can also make
$\Delta(C_*)$ multivalued.

Let $M$ be a closed, oriented rational homology 3-sphere with a
triangulation, or more generally a cellulation which may support a
combinatorial local system.  We first calculate its untwisted Reidemeister
torsion with coefficients in $F = \Q$.  Using the orientation, we can
canonically augment the chain complex $C_*(M;\Q)$ at both ends to obtain
the acyclic complex
\begin{eq}{e:augment} Q_* = \left\{ \begin{aligned} &0 \longto \Q
    \longto C_3(M;\Q) \longto C_2(M;\Q) \\
    &\quad \longto C_1(M;\Q) \longto C_0(M;\Q) \longto \Q \longto 0
    \end{aligned} \right\}. \end{eq}
Then it is standard that
\[ \Delta(Q_*) = \pm |H_1(M;\Z)|. \]
The sign is not a topological invariant, because the $j$-simplices of $M$
are unordered and unoriented, so they only provide $C_j(M;\Q)$ with an
unordered, unsigned basis.  We choose an ordering and an orientation of
the cells such that $\Delta(Q_*) > 0$.  We can then use the same ordering
and orientation for a let Reidemeister torsion calculation on $M$ with
twisted coefficients.

Suppose further that 
\[ H_1(M) = H_1(M;\Z) \cong \Z/n. \]
Then to compute the Reidemeister torsion of $M$, we let $F = \Q(\zeta_n)$,
where $\zeta_n$ is an abstract primitive $n$th root of unity, \ie,
an abstract root of the $n$th cyclotomic polynomial.  We also choose a
cellular cocycle $\omega \in C^1(M;\Z/n)$ such that $[\omega]$ generates
$H^1(M;\Z/n)$.  We use $\omega$ to define a twisted coefficient system
$\Q(\zeta)_\omega$ on $M$, and we let
\[ R_* \defeq C_*(M;\Q(\zeta_n)_\omega)\]
to define the Reidemeister torsion $\Delta(R_*)$ of $M$.  A change in the
choice of the generator $[\omega]$ can change $\Delta(R_*)$ by a Galois
automorphisms of $\Q(\zeta_n)$.  After fixing $[\omega]$, a change in the
choice of its representative $\omega$ can change $\Delta(R_*)$ by a factor
of $\zeta_n^c$ for some residue $c \in \Z/n$.  Otherwise $\Delta(R_*)$
is a topological invariant of $M$, provided that the cells of $M$ are
ordered and oriented so that $\Delta(Q_*) > 0$.

In particular, if $M = L(n,k)$, then
\begin{eq}{e:torsion}
\Delta(R_*) = \zeta_n^c (1-\zeta_n^a)(1-\zeta_n^b),
\end{eq}
where 
\[ \frac{a}{b} = k^{\pm 1} \in \Z/n. \]
This answer is easy to calculate using the standard cellulation of $L(n,k)$
with one cell in each dimension, as follows.  For a convenient choice of
twisted coefficients, this CW complex yields
\[ 0 \longto \Q(\zeta_n) \stackrel{1-\zeta_n}\longto \Q(\zeta_n)
    \stackrel{0}\longto \Q(\zeta_n) \stackrel{1-\zeta_n^k}\longto
    \Q(\zeta_n) \longto 0. \]
Thus,
\[ \Delta(R_*) = (1-\zeta_n)(1-\zeta_n^k). \]
The formula \eqref{e:torsion} is the same as this one, except generalized
to let $\Delta(R_*)$ change with a change in the choice of $\omega$.
The exponents $a$ and $b$ are also ambiguous, as follows.  First, the value
of the torsion \eqref{e:torsion} does not determine the global sign of $a$
and $b$, only their relative sign, since
\[ \zeta_n^c (1-\zeta_n^a)(1-\zeta_n^b) = \zeta_n^{a+b+c}
    (1-\zeta_n^{-a})(1-\zeta_n^{-b}). \]
The formula is also symmetric in $a$ and $b$, so we cannot distinguish $k$
from $1/k$. This stands to reason because
\[ L(n,k) \cong L(n,1/k). \]

\section{Proof of \Thm{th:main}}

To prove \Thm{th:main}, we begin with two basic results in numerical
algorithms.

\begin{theorem}[Edmonds \cite{Edmonds:distinct}] The determinant $\det M$
of a square matrix $M$ defined over $\Q(i)$, the field of complex numbers
with rational real and imaginary parts, can be computed in deterministic
polynomial time in the bit complexity of $M$.
\label{th:det} \end{theorem}

Edmonds states his result over an integral domain with suitable arithmetic
algorithms; the context of the paper suggests integer matrices.  However,
his construction works just as well using exact arithmetic in the field
$\Q(i)$.  He defines a variation of Gaussian elimination such that every
number that ever appears is a minor of the original matrix $M$.  As a result,
all numbers that arise in the calculation have polynomial bit complexity.

\begin{remark} There are many ways to prove \Thm{th:det} and we do not
know the best attribution.  The hard part of the result is to bound the
bit complexity of intermediate expressions, rather than just the number
of arithmetic operations.
\end{remark}

\Thm{th:det} is related to the problem of calculating the Smith normal
form of a matrix.

\begin{theorem}[Kannan-Bachem \cite{KB:smith}] The Smith normal form of
a square or rectangular matrix $M$ defined over $\Z$, together with left
and right multipliers, can be computed in deterministic polynomial time
in the bit complexity of the $M$.
\label{th:smith} \end{theorem}

Let $M$ be an oriented rational homology 3-sphere described by a
triangulation $\Theta$.  As a first step which will be important later,
we can simplify $\Theta$ to a cellulation $\Phi$ with one vertex and by
removing enough triangles until all of the tetrahedra merge into a single
3-cell, and dually by collapsing edges that connect two distinct vertices
until only one vertex is left.  Using either $\Theta$ or $\Phi$,
we can calculate the cellular chain complex $C_*(M;\Z)$ in polynomial
time.  We can use \Thm{th:det} to calculate the torsion $\Delta(Q_*)$
of the augmentation $Q_*$ in equation \eqref{e:augment}; in particular to
determine whether $C_*(M;\Z)$ has a positive or negative basis.  We can
assume a positive basis.

We can iteratively use \Thm{th:smith} to calculate a change of basis of
the chain complex $C_*(M;\Z)$ to put every differential $\del_k$ into
Smith normal form.  This also puts the dual complex $C^*(M;\Z)$ into Smith
normal form.  Using Smith normal form, if $H_1(M;\Z) \cong \Z/n$, then we can
calculate a cocycle $\omega \in C^1(M;\Z/n)$ that generates $H^1(M;\Z/n)$,
and we can express $\omega$ in the original basis of $C^*(M;\Z/n)$.

After calculating $\omega$, we can form the chain complex $R_*$ described
in \Sec{s:torsion}.  However, we will want to generalize the calculation,
and instead of computing torsion over the abstract field $\Q(\zeta_n)$ which
may have exponential dimension over $\Q$, we will compute it over the complex
numbers $\C$.  To this end, let $\zeta_n = \exp(2\pi i/n)$, and let $\zeta =
\zeta_n^\ell$ for certain exponents $\ell \in (\Z/n)^*$.  Note that $\ell$
need not be a prime residue, only non-zero, so $\zeta$ may have some lower
order $m|n$ with $m > 1$. Then we can form the chain complex $R_*(\zeta)$,
and its torsion has the same form as in equation \eqref{e:torsion}:
\begin{eq}{e:tor2}
\Delta(R_*(\zeta)) = \zeta^c(1-\zeta^a)(1-\zeta^b).
\end{eq}
Note that the constants $a$, $b$, and $c$ depend only on $\omega$ and not
on the exponent $\ell$.

If the cell complex $\Phi$ has $g$ edges, then it also has $g$ 2-cells,
and we can write the complex $R_*(\zeta)$ as
\[ 0 \longto \C \stackrel{\del_3}\longto \C^g
    \stackrel{\del_2}\longto \C^g \stackrel{\del_1}\longto
    \C \longto 0. \]
The complex $R_*(\zeta)$ is acyclic, so $\del_3$ is injective while $\del_1$
is surjective.   We can now make an adapted basis as follows:
\begin{enumerate}
\item We use the canonical basis vector $1 \in \C$ in degree 3 of
the chain complex $R_*(\zeta)$, and its image under $\del_3$ in degree 2.
\item We choose a non-zero entry of the vector $\del_3$.  If we choose
the $j$th entry $(\del_3)_j$, then we can omit the $j$th canonical basis
vector of $\C^g$ in degree 2.   We also use the image of these $g-1$
vectors under $\del_2$ in degree 1.
\item We choose a non-zero entry of the dual vector $\del_1$.  If we choose
the $k$th entry, then we include the $k$th basis vector of $\C^g$
in degree 1 and its image under $\del_1$, which is simply the scalar
value $(\del_1)_k$.
\end{enumerate}
Let $\del_2^{(j,k)}$ denote the matrix of $\del_2$ omitting the $j$th
column and the $k$th row.  Then we can express the Reidemeister torsion
of $R_*(\zeta)$ as
\[ \Delta(R_*(\zeta)) = \frac{(\del_3)_j(\del_1)_k}{\det \del_2^{(j,k)}}. \]

To compute $\Delta(R_*(\zeta))$ over $\C$, the most important question
is how many digits of precision we need throughout the calculation for an
accurate final answer.

\begin{lemma} Suppose $\zeta = \exp(2\pi i \ell/n) \in \C$ and that $R_*(\zeta)$
is the chain complex of $M$ with its local system $\C_\omega$.  Suppose that
we want to calculate $z \in \Q(i)$ such that
\[ \Delta(R_*(\zeta)) = z + O(n^{-\alpha}) \]
for some constant $\alpha$.  Then it suffices to calculate $\det
\del_2^{(j,k)}$ by estimating its entries with $d$ digits of precision,
where $d$ is polynomial in $\log(n)$, $g$, and $\alpha$.  Moreover, the
determinant can be calculated in polynomial time.
\label{l:precise} \end{lemma}

\begin{proof} Both $(\del_3)_j$ and $(\del_1)_k$ are of the form $\zeta^a
- \zeta^b$ for some constants $a$ and $b$, so each of these factors of
order $\Omega(\frac1n)$.  Thus we need to estimate $\det \del_2^{(j,k)}$
to a precision of $O(n^{-\alpha-2})$.  Each entry $\del_2^{(j,k)}$ is
$O(g)$, and therefore each $(g-1) \times (g-1)$ minor of the same matrix is
$O(g^{2g})$ since the determinant expansion has $g! = O(g^g)$ terms and each
term is $O(g^g)$.  So it suffices to estimate each of the $O(g^2)$ terms to
precision $O(n^{-\alpha-2}g^{-2g-2})$ in order for $\det \del_2^{(j,k)}$
(if it is then computed exactly) to have the desired accuracy.  Moreover,
each term is $O(g)$, which requires $O(\log(g))$ digits to the left of
each decimal point.  Thus the total  number of digits need to express each
entry is
\[ d = O(\log(g) + \log(n^{-\alpha-2}g^{-2g-2})), \]
which is polynomial in $\log(n)$, $\alpha$, and $g$.  We can then apply
\Thm{th:det} to exactly compute the determinant with these approximate
entries.
\end{proof}

To complete the proof of \Thm{th:main}, recall that $\zeta = \zeta_n^\ell$.
Recall from equation \eqref{e:tor2} that the Reidemeister torsion is
\[ f_-(\zeta) \defeq \Delta(R_*(\zeta)) = \zeta^c (1-\zeta^a)(1-\zeta^b). \]
We want to calculate several values of $f$ to obtained simplified sparse sums:
\begin{align*}
f_+(\zeta) &\defeq \frac{f_-(\zeta^2)}{f_-(\zeta)}
    = \zeta^c (1+\zeta^a)(1+\zeta^b) \\
g_+(\zeta) &\defeq \frac{f_+(\zeta)+f_-(\zeta)}2
    = \zeta^c + \zeta^{a+b+c} \\
g_-(\zeta) &\defeq \frac{f_+(\zeta)-f_-(\zeta)}2
    = \zeta^{a+c} + \zeta^{b+c} \\
h(\zeta) &\defeq \frac{g_+(\zeta)^2 - g_+(\zeta^2)}2 = \zeta^{a+b+2c}.
\end{align*}
At this point we assume that $n>4$, which we can do since otherwise
we can compute the Reidemeister torsion of $M$ directly over the field
$\Q(\zeta_n)$.  Using the three evaluations
\begin{eq}{e:eval} \Delta(R_*(\zeta_n)) \qquad \Delta(R_*(\zeta_n^2))
    \qquad \Delta(R_*(\zeta_n^4)), \end{eq}
we can learn the sum $g_+(\zeta_n)$ and the product $h(\zeta_n)$ of
$\zeta_n^c$ and $\zeta_n^{a+b+c}$; and the sum $g_-(\zeta_n)$ and the
product $h(\zeta_n)$ of $\zeta_n^{a+c}$ and $\zeta_n^{b+c}$.  We can thus
solve quadratic equations to obtain all four of these numbers.  We can
then learn the values of an unordered pair $\{\zeta_n^{sa},\zeta_n^{sb}\}$,
where $s = \pm 1$, by taking ratios.

If we calculate the torsion values \eqref{e:eval} using floating point
arithmetic over $\C$, we obtain floating point approximations to
\[ \zeta^{sa} = \exp(\frac{2\pi i sa}n) \qquad
    \zeta^{sb} = \exp(\frac{2\pi i sb}n). \]
We can then numerically calculate logarithms to obtain the arguments
$2\pi sa/n$ and $2\pi sb/n$.  If at this point we know $z$ and $\arg(z)$
to $O(\log(n))$ digits of precision, we can calculate the residues $sa,
sb \in \Z/n$ by rounding their computed values to the nearest integer.
We can then take their ratio in the ring $\Z/n$ to obtain $k^{\pm 1}$.
Working backwards, we it suffices to compute each values of $f_+$, $g_\pm$,
and $h$ with $O(\log(n))$ digits of precision.  We can use \Lem{l:precise}
to specify the precision at the beginning of the calculation in order to
have enough precision at this last stage.


\providecommand{\bysame}{\leavevmode\hbox to3em{\hrulefill}\thinspace}
\providecommand{\MR}{\relax\ifhmode\unskip\space\fi MR }
\providecommand{\MRhref}[2]{%
  \href{http://www.ams.org/mathscinet-getitem?mr=#1}{#2}
}
\providecommand{\href}[2]{#2}
\providecommand{\eprint}{\begingroup \urlstyle{tt}\Url}

\end{document}